\documentclass[preprint]{elsarticle}

\usepackage{amssymb, amsmath, amsthm, amscd, url}
\usepackage{graphicx}
\usepackage{graphicx}
\input epsf

\usepackage{subfigure}
\begin{document}
\newcommand{\tsum}{$T=T_1+_{\Delta} T_2$}
\newcommand{\bdy}{\partial}
\newcommand{\thh}{T(1/2, -1/2)}
\newcommand{\cont}{\subset}
\newcommand{\bpann}{$\partial \Delta$-anannular}
\newcommand{\Int}{{\rm Int}}

\title{Hyperbolicity of arborescent tangles and arborescent links}
\author{Kathleen Reif Volz}
\ead{kathleenvolz@augustana.edu}
\address{Department of Mathematics, University of Iowa, Iowa City, IA 52242, U.S.A.}
\address{Current Address:  Mathematics/Computer Science Department, Augustana College \\ 639  38th Street, Rock Island, IL 61201, U.S.A. }

\begin{keyword}
Arborescent tangles \sep arborescent links \sep hyperbolic manifolds

\MSC{\em Primary 57N10}
\end{keyword}

\begin{abstract}
In this paper, we study the hyperbolicity of arborescent tangles and arborescent links.  We will explicitly determine all essential surfaces in arborescent tangle complements with non-negative Euler characteristic, and show that given an arborescent tangle $T$, the complement $X(T)$ is non-hyperbolic if and only if $T$ is a rational tangle, $T=Q_m \ast T'$ for some $m \geq 1$, or $T$ contains $Q_n$ for some $n \geq 2$.  We use these results to prove a theorem of Bonahon and Seibenmann which says that a large arborescent link $L$ is non-hyperbolic if and only if it contains $Q_2$.
\end{abstract}

\newtheorem{thm}{Theorem}[section]
\newtheorem{lem}[thm]{Lemma}
\newtheorem{con}[thm]{Conjecture}
\newtheorem{cor}[thm]{Corollary}
\newtheorem{prop}[thm]{Proposition}
\newtheorem{addend}[thm]{Addendum}
\newtheorem{defin}[thm]{Definition}
\newtheorem{example}[thm]{Example}
\newtheorem*{claim}{Claim}  
\newtheorem*{rem}{Remark}
\newtheorem*{ntbd}{Note (to be deleted)}

\maketitle

\section{Introduction}
\label{section: Intro}

Arborescent tangles were defined by Conway \cite{Conway} as tangles which can be obtained from the trivial tangles by certain operations.  He used these to study a class of links which he called algebraic links.  His purpose was to generalize 2-bridge links, also called rational links.
Rational tangles make up the most basic class of such tangles; every rational tangle is associated with a unique rational number, $p/q$, or $\infty$, and Conway was the first to note that two rational tangles are isotopic if and only if they correspond to the same rational number.  Later Gabai named Conway's algebraic links \emph{arborescent links} because the name algebraic links had already been used before Conway for another class of links.  Arborescent links have also been studied by Montesinos \cite{Montesinos}, Hatcher and Thurston
\cite{HatcherThurston}, Oertel \cite{Oertel}, and many others.

 
Since arborescent tangles (resp. links) are built up from rational tangle components, we often want to decompose a tangle (link) into two arborescent tangle pieces.  This involves cutting along a decomposing disk (sphere) called a Conway disk (Conway sphere), which cuts the tangle or link into a set of rational tangles.  The length of an arborescent tangle or a large arborescent link is defined to be the minimum number of rational tangles among all such decompositions.

Wu classified all arborescent tangles without closed components whose exteriors are
hyperbolic in the sense that such a tangle admits a hyperbolic structure with totally geodesic boundary \cite{Wu}.  The main purpose of this paper is to study the same problem for the complement of arborescent tangles, allowing closed components.  Given an arborescent tangle $(B,T)$, define the tangle complement to be $X(T)=B-T$, and the tangle exterior $E(T)=B-\Int N(T)$.  Let $Q_m$ be the tangle with two vertical strings and $m$ horizontal circles, as shown in Figure \ref{fig: Q3 and annulus and torus}.  Given two tangles $T_1, T_2$, define $T_1 \ast T_2$ to be the tangle obtained by gluing $T_1$ on top of $T_2$.  See the paragraph before Definition  \ref{def: Q_n tangle, std torus, etc} for more details.  We can now state the main theorem from Section \ref{section: Hyperbolicity of complements}.

\bigskip

\noindent \textbf{Theorem \ref{thm: hyperbolic tangle complements}.} 
\textit{Suppose $T$ is an arborescent tangle.  Then $X(T)$ is non-hyperbolic if
and only if one of the following holds.}

\textit{(1) $T$ is a rational tangle.}

\textit{(2) $T = Q_m * T'$ for some $m\geq 1$.}

\textit{(3) $T$ contains $Q_n$ for some $n\geq 2$.}

\bigskip

A {\it standard annulus} in $Q_m$ is an annulus separating the circles from the two vertical arcs.  Similarly for {\it standard torus}.  See Definition \ref{def: Q_n tangle, std torus, etc} for more details.  The tangle complement $X(T)$ is non-hyperbolic if and only if it contains an essential surface $F$ which is a sphere, disk, annulus, or torus.  These can be determined explicitly as follows.

\bigskip

\noindent \textbf{Addendum \ref{addendum from YQW}.}  
\textit {Suppose $T$ is an arborescent tangle.}

\textit {(1) $X(T)$ contains no essential $S^2$.}

\textit {(2) $X(T)$ contains an essential disk $D$ if and only if $T$ is a rational tangle and $D$ is the disk separating the two strings of $T$.}

\textit {(3) $X(T)$ contains an essential annulus $A$ if and only if $T = Q_m * T'$ for some $m\geq 1$ and $A$ is a standard annulus in $Q_m$.}

\textit  {(4) $X(T)$ contains an essential torus $F$ if and only if $T$ contains a $Q_m$ for some $m\geq 2$ and $F$ is a standard torus in $Q_m$.}

\bigskip

Bonahon-Siebenmann classified all non-hyperbolic arborescent links in an unpublished manuscript \cite{BS}.  Oertel studied Montesinos links and found exactly which ones are hyperbolic.  See Theorem \ref{thm: Oertel's Montesinos hyperbolic links} for his statement.  We will use Theorem \ref{thm: hyperbolic tangle complements} to give a proof of the following theorem.  Together with Oertel's theorem, this gives a complete proof of Bonahon-Seibenmann's theorem for the classification of non-hyperbolic arborescent links.  

\bigskip

\noindent \textbf{Theorem \ref{thm: Bonahon-Seibenmann large arborescent links}.  {\rm (Bonahon-Seibenmann)}} 
\textit{ Suppose $L$ is a large arborescent link.  Then $L$ is non-hyperbolic if and only if it contains $Q_2$.}

\bigskip

An alternative proof of Bonahon-Seibenmann's theorem has been given by Futer and Gueritaud \cite{FuterGueritaud}, using a different method.

Gabai's definition for arborescent links uses tree diagrams (hence the use of the Latin word \emph{arbor}, meaning tree).   In this paper we define an arborescent link to be a Montesinos link or a link obtained by gluing two non-trivial arborescent tangles to each other.  See Definition \ref{defin: arborescent links in terms of tangles}.  The two definitions are equivalent for prime links, as shown in \cite{Reif}.  
We will also show that if $L$ is a large arborescent link then it is also prime.  See Theorem \ref{thm: large arborescent link is prime}.



\section{Definitions and Preliminaries}
\label{section: Definitions and Preliminaries}
 
Unless otherwise stated, in this paper surfaces are compact and orientable, and surfaces in 3-manifolds are properly embedded.  A surface $F$ in a 3-manifold $M$ is essential means it is incompressible, $\bdy$-incompressible, and not $\bdy$-parallel.  The manifold $M$ is $\bdy$-irreducible means $\bdy M$ is incompressible in $M$.  Given a set $X$ in a manifold $M,$ let $N(X)$ denote a regular neighborhood of $X$ in $M$.  We use $A \parallel B$ to denote that $A$ is parallel to $B$.   Other classical definitions can be found in Hempel \cite{Hempel}, Jaco \cite{Jaco}, or Hatcher's notes \cite{Hatcher2}.

A \emph{tangle} is a pair, $(B,T)$, where $B$ is a 3-ball and $T$ is a properly embedded 1-manifold.  In this paper we always assume that $T$ consists of 2 arcs and possibly some circles, so $T$ intersects $\bdy B$ in exactly 4 points.  A \emph{marked tangle} is a triple $(B, T, \Delta)$ where $(B,T)$ is a tangle and $\Delta$ is a disk on $\bdy B$ containing exactly two endpoints of $T$, called the {\it gluing disk}.  We use $T$ to describe a tangle when $B$ and $\Delta$ are not ambiguous.  

\begin{figure}[hbt]
\centering
\includegraphics[width=4in]{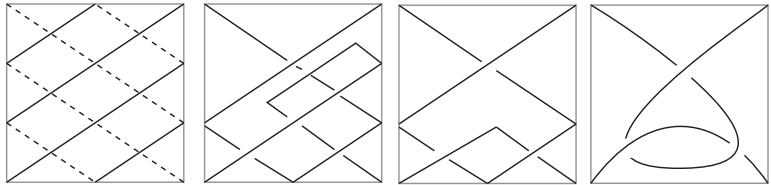}
\caption{The tangle $T[2/3],$ first inscribed on a pillowcase and then progressively simplified.}
\label{fig: T23}
\end{figure}

A \emph{rational tangle $T[p/q]$} is a tangle drawn by inscribing lines with slope $p/q$ on a ``pillowcase'' with four holes at the corners.   Figure \ref{fig: T23} gives an example of the simplification of the rational tangle $T[2/3]$ starting with the tangle drawn on a pillowcase.  The class of rational tangles includes the two trivial tangles, $T[0]$ and $T[\infty]$, as shown in Figure \ref{fig: trivial tangles}.  

\begin{figure}[hbtp]
\centering
\includegraphics[width=3in]{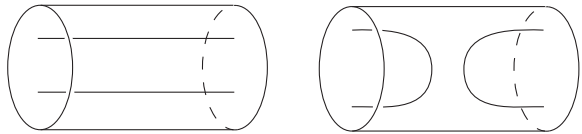}
\caption{The tangle on the left is T[0]; on the right is T[$\infty$].}
\label{fig: trivial tangles}
\end{figure}

Given a rational tangle $(B,T)$ in standard position (as drawn on the pillowcase $\bdy B)$, define a \emph{horizontal circle} as a simple closed curve on $\bdy B$ running horizontally and a \emph{vertical circle} as a simple closed curve on $\bdy B$ running vertically.  For example, the equator is a horizontal circle.

Given a tangle $(B, T)$, define the tangle complement to be
$X(T) = B - T$, and the tangle exterior to be $E(T) = B - \Int N(T)$.  While
they are homotopic, it is important to note that there are major
differences between a surface in $X(T)$ and a surface in $E(T)$.  For example,
the boundary of $X(T)$ is a 4-punctured sphere $\bdy B - T$, while the
boundary of $E(T)$ is a genus 2 surface.  Also, a surface properly
embedded in $X(T)$ may be $\bdy$-compressible in $E(T)$ but not in
$X(T)$.  

Given a string $t_i$ from a tangle $T$, the exterior of the string
$t_i$ is denoted $E(t_i)$, i.e. $E(t_i)= B-\Int N(t_i)$.  While $\bdy
M$ usually denotes the boundary of the 3-manifold $M$, it is
convenient to use the notation $\bdy N(T)$ to denote the
\emph{frontier} of the regular neighborhood $N(T)$ of $T$ instead of
the whole boundary of $N(T)$.  For example if $T$ is a pair of
arcs then $\bdy N(T)$ is a pair of annuli.  Similarly, we define $\bdy
N(t_i)$ to be the frontier of $N(t_i)$ when $t_i$ is a string of $T$.

Two marked tangles, $(B_1, T_1, \Delta_1)$ and $(B_2, T_2, \Delta_2)$,
are \emph{equivalent} means there is an orientation preserving homeomorphism of triples from
$(B_1, T_1, \Delta_1)$ to $(B_2, T_2, \Delta_2)$.  For example we can
see that two rational tangles $T[p_1/q_1]$ and $T[p_2/q_2]$ are
equivalent if and only if $p_1/q_1 \equiv p_2/q_2$ mod $\mathbb{Z}$.
Given two tangles, $(B_1, T_1, \Delta_1)$ and $(B_2, T_2, \Delta_2)$,
the sum $(B_1, T_1, \Delta_1) + (B_2, T_2, \Delta_2)$ is defined by
choosing a gluing map $\phi: \Delta_1 \rightarrow \Delta_2$ with
$\phi(\Delta_1 \cap T_1)=\Delta_2 \cap T_2;$ write $(B,T)=(B_1, T_1,
\Delta_1) + (B_2, T_2, \Delta_2)$ and say that $(B,T)$ is the sum of
the two tangles.  More simply, we say $T=T_1 + T_2$.  Note that this
sum depends on the gluing map $\phi$, but in most cases the property
of $T$ in which we are interested is not affected by the choice of
$\phi$.  In the case that we want to be specific about the gluing
disk, $\Delta \approx \Delta_1 \approx \Delta_2$, we denote the sum as
$T= T_1 +_{\Delta} T_2$.  Furthermore, we define the twice-punctured
disk $P(\Delta)=\Delta \cap X(T_1)=\Delta \cap X(T_2)$.  A sum of two
marked tangles, \tsum, is called nontrivial exactly when neither $(B_1, T_1,
\Delta _1)$ nor $(B_2, T_2, \Delta _2)$ is $T[0]$ or $T[\infty]$.

An arborescent tangle $T$ can be defined in terms of rational tangles as follows.  Rational tangles are arborescent tangles, and any nontrivial sum of arborescent tangles is an arborescent tangle.  Arborescent links are built from these arborescent tangles and will be explained more below.

Montesinos tangles are a smaller class within the class of arborescent tangles. They are characterized by the fact that their gluing disks are mutually disjoint.  Tangles written in the form $T(r_1, r_2, ..., r_n)$ with $r_i$ a rational number for $i=1, .., n$ are Montesinos tangles drawn by connecting each tangle $T[r_i]$ in order from left to right and connecting the top strings and bottom strings.  See Figure \ref{fig: Montesinos Tangle} for a diagram.  

\begin{figure}[hbt]
\centering
\includegraphics[width=3in]{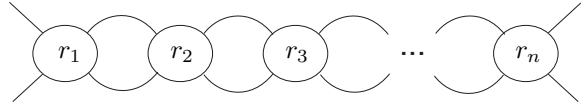}
\put(-198,17){$r_1$}
\put(-155.5,17){$r_2$}
\put(-112,17){$r_3$}
\put(-25,17){$r_n$}
\caption{A Montesinos Tangle}
\label{fig: Montesinos Tangle}
\end{figure}

The \emph{length of an arborescent tangle $T$}, given by $\ell (T)$, is the minimum number of (non-trivial) rational tangles from which $T$ can be written as a sum.  An example is shown in Figure \ref{fig: length}.  

\begin{figure}[hbtp]
\centering
\includegraphics[width=1.5in]{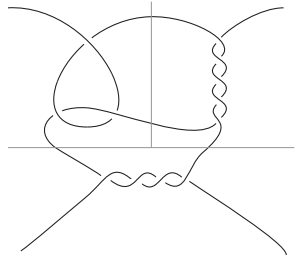}
\caption{An arborescent tangle with length 3.}
\label{fig: length}
\end{figure}

As defined by Gabai \cite{Gabai}, an arborescent link is the boundary of a surface constructed by plumbing (Murasugi sum along a 4-gon) as specified by a tree.  The reader is referred to Gabai's paper \cite{Gabai} for details on how the trees relate to the links.  An
example of such a tree and associated link is shown in Figure \ref{fig: arborescent tree diagram}.

\begin{figure}[hbtp]
\centering
\includegraphics[width=4in]{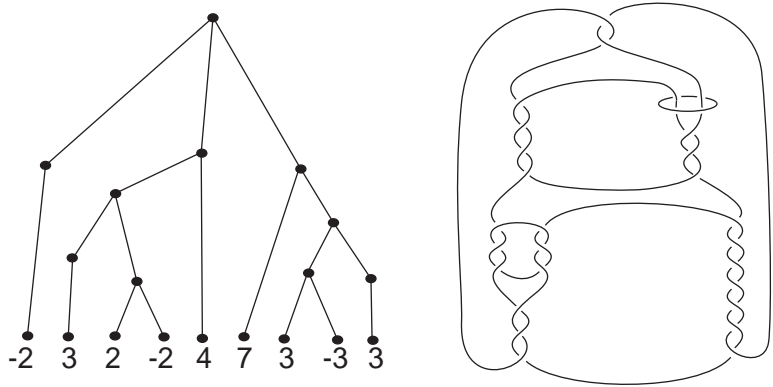}
\caption{An example of a tree diagram (left) and its associated arborescent link (right).}
\label{fig: arborescent tree diagram}
\end{figure}

A \emph{Conway sphere} for a link $L$ in $S^3$ is a sphere $S$ intersecting $L$ at 4 points, such that $S - L$ is incompressible in $S^3 - L$.  Similarly, a \emph{Conway disk} for a tangle $(B, T)$ is a disk $D$ in $B$ intersecting $T$ at two points, such that $D - T$ is
incompressible in $B-T$, and there is no disk $E$ in $B-T$ with $\bdy E$ a union of two arcs $\alpha \cup \beta$, where $\alpha \subset \bdy B$, and $\beta = E \cap D$ is an essential arc on $D - T$.  A Conway disk will also be called a \emph{decomposing disk.}

An arborescent link can also be defined in terms of arborescent tangles. This is more convenient for our purposes.  If $T = T(p_1/q_1, ..., p_n/q_n)$ is a Montesinos tangle and $q_i > 1$ for all $i$, then the numerator closure of $T$ is called a Montesinos link of length $n$. (See Figure \ref{fig: Montesinos Link}.)  Note that a Montesinos link
of length 1 or 2 is a 2-bridge link.  Montesinos links have been studied in detail by Oertel \cite{Oertel}, who called them star links since the tree diagrams (as in Gabai \cite{Gabai}) are star-shaped.  To denote a Montesinos link as in Figure \ref{fig: Montesinos Link}, we write $L(r_1, r_2,...,r_n)$, where $r_i=p_i/q_i$.  Burde and Zieschang's book \cite{BurdeZieschang} gives more detail on Montesinos Links.

\begin{defin} {\rm (Wu \cite{Wu2})}
\label{defin: arborescent links in terms of tangles}
A link is a \emph{small arborescent link} if it is a Montesinos link
of length at most $3$.  A link is a \emph{large arborescent link} if
it has a Conway sphere cutting it into two non-rational arborescent
tangles.  A link is an \emph{arborescent link} if it is either a small
arborescent link or a large arborescent link.
\end{defin}

There is a slight difference between the definition for arborescent links given above and Gabai's tree definition.  For example, the diagram shown in Figure \ref{fig: NOT arborescent link} is not arborescent by the above definition although it can be obtained from a
tree diagram as defined by Gabai if an end-vertex with zero weight is allowed.  However, notice that this link is a composite link.  
In this paper, we use the tangle-definition of arborescent link.  This will not cause loss of generality when studying the hyperbolicity of arborescent links because we already know that composite links are non-hyperbolic.

\begin{figure}[hbtp]
\centering
\includegraphics[width=4in]{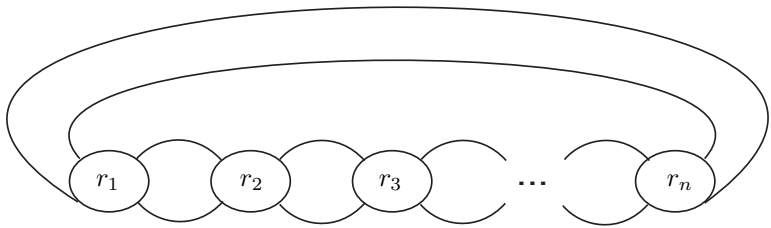}
\put(-255,15){$r_1$}
\put(-200.5,15){$r_2$}
\put(-148,15){$r_3$}
\put(-39,15){$r_n$}
\caption{A Montesinos link $L(r_1, r_2, r_3, ..., r_n)$.}
\label{fig: Montesinos Link}
\end{figure} 

\begin{figure}[hbtp]
\centering
\includegraphics[width=2.5in]{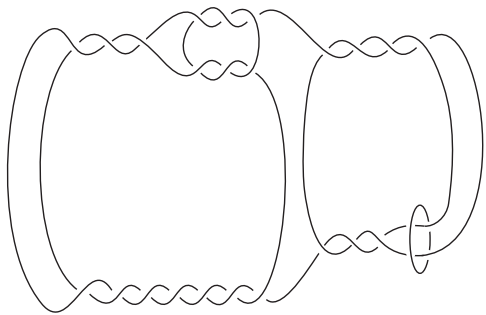}
\caption{This is an example of a link which is NOT an arborescent link
  by the tangle definition.  However, if end-vertices with weight zero
  are allowed in Gabai's definition, then this link is arborescent by
  that definition.}
\label{fig: NOT arborescent link}
\end{figure}

\section{Hyperbolicity of Tangle Complements}
\label{section: Hyperbolicity of complements}

Recall that the complement of a tangle $T = (B,T)$ is the non-compact
manifold $X(T) = B - T$.  

\begin{defin} The complement $X(T)$ of a tangle $T$ is \emph{hyperbolic}
if it is irreducible, $\bdy$-irreducible, atoroidal, and anannular.  
\end{defin}

If $X(T)$ is hyperbolic by the above definition then the double of
$X(T)$ along $\bdy X(T)$ with toroidal cusps removed is a compact
manifold with toroidal boundary, which is irreducible, atoroidal, and cannot be Seifert fibered because $\bdy X(T)$ is a separating incompressible surface with negative Euler characteristic. (See the proof of Lemma \ref{lem: Non Seifert fibered}.)
Therefore the double of $X(T)$ is hyperbolic, and hence $X(T)$ admits
a complete hyperbolic structure with totally geodesic boundary (see Thurston \cite{Thurston}).  The main theorem of this paper is Theorem \ref{thm: hyperbolic tangle
  complements}, which determines all non-hyperbolic arborescent tangle
complements.

Proposition \ref{prop: P essential and bdy sphere incompressible}
shows that if $X(T)$ is $\bdy$-reducible then $T$ is a rational
tangle.  Proposition \ref{prop: arb tangle complements irred} shows that
$X(T)$ is always irreducible.  Proposition \ref {prop: essential
  annulus is standard} determines all $X(T)$ which contain essential
annuli, and Proposition \ref {prop: toroidal tangle contains Q2 with
  torus standard} determines those containing essential tori.  
Theorem \ref{thm: hyperbolic tangle complements} follows from these
propositions.  

\subsection{Essential disks and essential spheres in $X(T)$}


Given a rational tangle $(B,T)=T[p/q]=(t_1 \cup t_2)$, a compressing
disk for $\bdy B -T$ separates the strings $t_1$ and $t_2$. If $T$ is
the trivial tangle $T[0]$, one can see that the horizontal disk with a
horizontal circle as its boundary is the only compressing disk for
$\bdy B-T$ up to isotopy.  Similarly, up to isotopy the only
compressing disk for the trivial tangle $T[\infty]$ is the compressing
disk with a vertical circle as its boundary.  Since any rational
tangle $T=T[p/q]$ is homeomorphic to a trivial tangle, we can see that
up to isotopy there is only one compressing disk for $\bdy X(T)$.
This fact will be used in the proof of Lemma \ref{lem: pq tangle}.

\begin{defin}
  Let $\alpha_1,$ $\alpha_2$ be simple closed curves on a surface $F$
  such that $\alpha_1 \cap \alpha_2 \neq \emptyset$.  A \emph{bigon
    between $\alpha_1,$ $\alpha_2$} is a disk $D \cont F$ such that
  $\bdy D=\alpha_1' \cup \alpha_2'$, where $\alpha_i'$ is an arc on
  $\alpha_i$.
\end{defin}

The following Lemma is from Casson and Bleiler \cite[pp. 26-30]{CassonBleiler}.

\begin{lem}
Suppose $\alpha,$ $\beta$ are simple closed curves on a compact hyperbolic surface $F$ such that there is no bigon between $\alpha$ and $\beta$.  Let $\beta'$ be a simple closed curve on $F$ which is isotopic to $\beta$.  Then $|\alpha \cap \beta| \leq |\alpha \cap \beta'|$, and equality holds iff there is no bigon between $\alpha$ and $\beta'$.  
\label{lem: casson and bleiler}
\end{lem}

\begin{rem}
Lemma \ref{lem: casson and bleiler} also holds for non-hyperbolic surfaces.
\end{rem}

\begin{lem}
Suppose $T$ is a $p/q$ rational tangle, $(p,q)=1,$ $q \geq 1,$ and let $S=\bdy B - T$.  Let $P,Q \cont \bdy B$ be twice punctured disks such that $\alpha =\bdy P=\bdy Q=P \cap Q$ is a vertical circle and $P$ is the left disk.  If $D$ is a compressing disk for $S$ and neither $D \cap P$ nor $D \cap Q$ contains an arc which is inessential on $P$ or $Q,$ respectively, then $|\alpha \cap \bdy D|=2q$.  In particular, $P$ is incompressible. 
\label{lem: pq tangle}
\end{lem}

\begin{proof}
Let $\alpha$ be a vertical circle on $\bdy B$.  By definition, $T$ is isotopic rel $\bdy T$ to a pair of arcs $c_1 \cup c_2$ of slope $p/q$ on the pillowcase $\bdy B$.  Note that $|c_i \cap \alpha|=q$.  Let $\beta$ be the boundary of a regular neighborhood of $c_1$ on $\bdy B$.  Then $\beta$ bounds a compressing disk of $\bdy B -T$ in $B-T$.  Figure \ref{fig: pq disk} demonstrates an example of such a compressing disk.

Since $|c_1 \cap \alpha|=q,$ we have $|\beta \cap \alpha| = 2q$.  Note that $\beta$ intersects $P$ and $Q$ in essential arcs, hence there is no bigon between $\alpha$ and $\beta$.  By the discussion above, $\bdy D$ is isotopic to $\beta,$ so by Lemma \ref{lem: casson and bleiler}, $|\bdy D \cap \alpha|=|\beta \cap \alpha|=2q$ if and only if there are no bigons between $\bdy D$ and $\alpha,$ i.e., each component of $\bdy D \cap P$ and $\bdy D \cap Q$ is essential. 
\end{proof}

\begin{figure}[hbtp]
\centering
\includegraphics[width=1.5in]{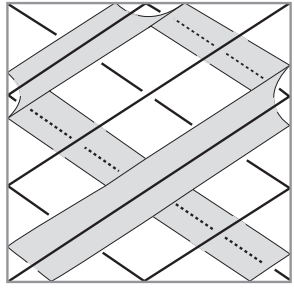}
\caption{The shaded region represents the compressing disk (up to isotopy) for the rational tangle $T(2/3)$.}
\label{fig: pq disk}
\end{figure}

\begin{lem}
Suppose $F$ is an essential surface in a $3$-manifold $M$.  If $M$ is $\bdy$-reducible then there exists a $\bdy$-reducing disk $D$ such that $D \cap F=\emptyset$.  If $M$ is reducible then there is a reducing sphere $S$ such that $S \cap F=\emptyset$.
\label{lem: boundary reducing disk with empty intersection AND sphere}
\end{lem}

\begin{proof}
This is a standard innermost circle/outermost arc argument.  Choose a $\bdy$-reducing disk $D$ so that the number of components $|D \cap F|$ is minimal.  (The proof for the reducing sphere is similar.)  If $D \cap F$ has inessential circle components on $F,$ let $D'$ be a disk on $F$ bounded by an innermost such component and let $D_1$ be the disk on $D$ bounded by $\bdy D'$.  Then $D_1'=(D-D_1) \cup D'$ can be perturbed so that $|D_1' \cap F| < |D \cap F|,$ contradicting the fact that $|D \cap F|$ is minimal.  If $D \cap F$ has some circle components and if they are all essential on $F,$ then a disk on $D$ bounded by an innermost circle component of $D \cap F$ would be a compressing disk of $F,$ contradicting the incompressibility of $F$.  Therefore $D \cap F$ has no circle components.  Similarly if $D \cap F$ has a trivial arc on $F$ then an outermost such arc $\alpha$ on $F$ would cut off a disk $D'$ on $F$ and $\alpha$ splits $D$ into $D_1$ and $D_2;$ at least one of the $D_i'=D_i \cup D'$ is then a $\bdy$-reducing disk of $M,$ which can be isotoped to reduce $|D \cap F|$.  If $D \cap F$ consists of essential arcs on $F$ then a disk on $D$ cut off by an outermost component of $D \cap F$ on $D$ would be a boundary compressing disk of $F,$ contradicting the $\bdy$-incompressiblity of $F$.  
\end{proof}

\begin{prop}
Let $T = T_1 +_{\Delta} T_2$ be a non-trivial sum of arborescent tangles.  Then $P(\Delta)$ is essential in $X(T)$ and $\bdy X(T)$ is incompressible in $X(T)=B-T$.  
\label{prop: P essential and bdy sphere incompressible}
\end{prop}

\begin{proof}
Let $T$ be a minimal counterexample in the sense that $T$ is an arborescent tangle such that $\ell(T)=n$ and there does not exist an arborescent tangle $T'$ such that $\ell(T') < n$ and $T'$ is a counterexample.  Then \tsum\ is a nontrivial sum and $n \geq 2$.  Let $P=P(\Delta)$.  We need to prove that $P$ is incompressible and $\bdy$-compressible, and $\bdy X(T)$ is incompressible.

Suppose $P$ is compressible.  Let $D$ be a compressing disk for $P,$ so $D \cont X(T_j)$ for $j=1$ or $2$.  If $\ell(T_j)=1$ then by Lemma \ref{lem: pq tangle}, $T_j=T[\infty]$.  This contradicts the fact that \tsum\ is a non-trivial sum.  If $\ell(T_j) >1,$ then $D$ is a compressing disk for $\bdy B-T_j$.  Since $\ell(T_j) < n,$ this contradicts the fact that $T$ is a minimal counterexample.  Therefore, $P$ is incompressible.

Next, note that a $\bdy$-compressing disk for $P$ is also a compressing disk for $\bdy B-T_j$ for $j=1$ or $2$.  Thus if there exists a $\bdy$-compressing disk for $P,$ we will have the same contradiction as above unless $T_j$ is a rational tangle.  In this case, $T_j$ must be an integral tangle and the sum is trivial.  Hence $P$ is $\bdy$-incompressible.

Finally, suppose $D'$ is a compressing disk for $\bdy B - T$.  Since $P$ is essential, by Lemma \ref{lem: boundary reducing disk with empty intersection AND sphere}, we can find a compressing disk $D''$ such that $D'' \cap P = \emptyset$.  Thus $D''$ is a compressing disk for $\bdy B_k-T_k$ for $k=1$ or $2,$ and $\ell(T_k)<n$.  Since $T$ is a minimal counterexample, $T_k$ cannot be a nontrivial tangle sum, so it is a rational tangle $T(p/q)$.  Since $D''$ is disjoint from $P$, by Lemma \ref{lem: pq tangle}, we must have $q=0$, so $T_k$ is a trivial tangle.  This contradicts the assumption that $T=T_1 +_{\Delta} T_2$ is a nontrivial sum.
\end{proof}

\begin{prop}
Arborescent tangle complements are irreducible.
\label{prop: arb tangle complements irred}
\end{prop}

\begin{proof}
If $T$ is rational then $E(T)$ is a handlebody and hence $X(T)$ is irreducible.  Suppose the result holds for any arborescent tangle $T$ such that $\ell(T) \leq n$.  Suppose $T'$ is an arborescent tangle such that $\ell(T') =n+1$ and suppose there exists an essential sphere, $S \cont X(T')$.  Write $T'=T_1 +_{\Delta} T_2$, where the sum is non-trivial.  By Lemma \ref{lem: boundary reducing disk with empty intersection AND sphere} and Proposition \ref{prop: P essential and bdy sphere incompressible}, we can find an essential sphere $S'$ which does not intersect $P(\Delta)$.  However, by the inductive hypothesis, this cannot happen.
\end{proof}

\subsection{Standard torus and standard annulus in $Q_n$}
\label{subsection: std torus and annulus}

\begin{defin}
A curve $\alpha$ on a planar surface $F$ is of \emph{type I} (resp. \emph{type II}) means it bounds a once-punctured (resp. twice-punctured) disk on $F$.   Suppose $A \cont E(T)$ is an annulus with $\bdy A \cont \bdy B$.  We call $A$ a \emph{type I annulus} (resp \emph{type II annulus}) exactly when $\bdy_i A$ is a type I (resp type II) curve on $\bdy B-T$ for $i=1,2$.  Furthermore, we note that there are two kinds of type I annuli.  We call $A$ a \emph{type I-A annulus} exactly when $A$ is a type I annulus such that $\bdy_1 A \parallel \bdy_2 A$ on $\bdy B-T$; We say that $A$ is a \emph{type I-B annulus} exactly when $\bdy_1 A \not \parallel \bdy_2 A$ on $\bdy B-T$.   (See Figure \ref{fig: annulus type}.)
\end{defin}

\begin{figure}[htbp]
\centering
\includegraphics[width=3.5in]{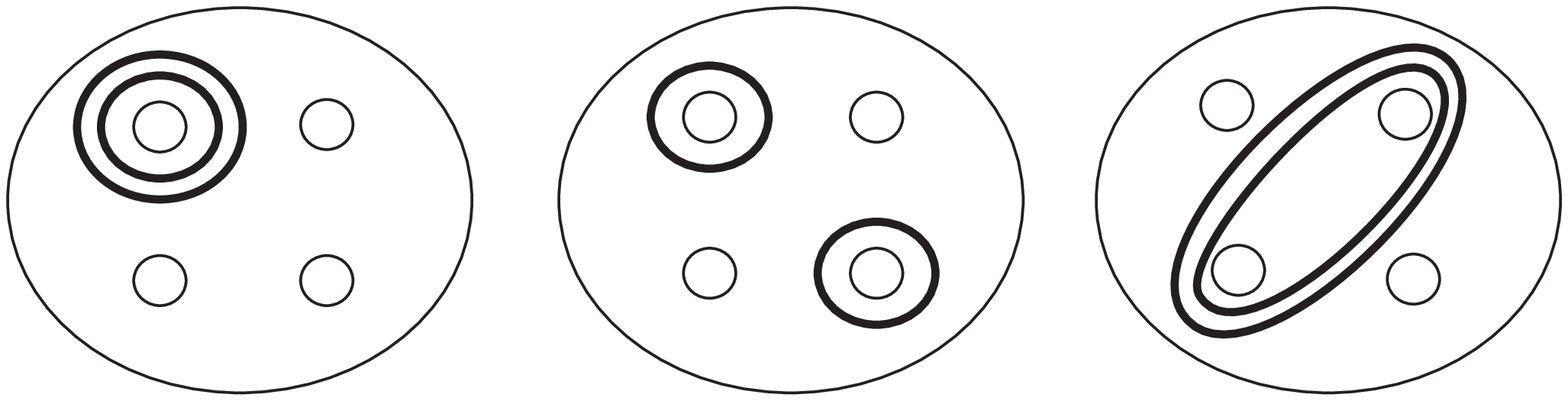}
\put(-245,-10){Boundary of a} \put(-250, -20){type IA annulus}
\put(-155,-10){Boundary of a} \put(-160, -20){type IB annulus}
\put(-65,-10){Boundary of a} \put(-70, -20){type II annulus}
\caption{The boundary of the different types of annuli as viewed on $S=\bdy B-T$.}
\label{fig: annulus type}
\end{figure}

\begin{lem}
  Suppose \tsum\ is an arborescent tangle, $A$ is an annulus in $X(T)$
  with $\bdy A \cont (\bdy B - T)$, and $\bdy A \cap \bdy \Delta =
  \emptyset$.  Then $(A \cap P(\Delta)) \sqcup (A \cap (\bdy B-T))$ cannot contain curve components of both types I and II.
\label{lem: can't have intersections of types I and II}
\end{lem}

\begin{proof}
Suppose we have an annulus $A$ such that $(A \cap P(\Delta)) \sqcup (A \cap (\bdy B-T))$ contains curves $\alpha_1$ and $\alpha_2$ with $\alpha_1$ a type I curve and $\alpha_2$ a type II curve.  By an innermost disk/outermost arc argument, we may assume that $\alpha_1$ and $\alpha_2$ are disjoint essential circles on $A$.  Thus, there must be an annulus $A' \cont A$ such that $\alpha_1=\bdy _1 A'$ bounds a disk $D_1$ on $\Delta \cup \bdy B$ that intersects the tangle in one point and $\alpha_2=\bdy _2 A'$ bounds a disk $D_2$ on $\Delta \cup \bdy B$ that intersects the tangle in two points (see Figure \ref{fig: annulusT12}).  Thus we have a sphere $S=(D_1 \cup A' \cup D_2) \cont B$ that intersects the tangle in three points. This is impossible.
\end{proof}

\begin{figure}[htbp]
\centering
\includegraphics[width=2in]{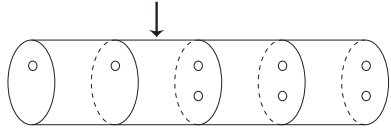}
\caption{The arrow points to $S=D_1 \cup A' \cup D_2$.}
\label{fig: annulusT12}
\end{figure}

Note that in the proof, $T_1$ could be trivial, in which case Lemma \ref{lem: can't have intersections of types I and II} says that every annulus $A \cont E(T)$ with $\bdy A \cont \bdy B-T$ such that $\bdy_i A$ does not bound a disk on $\bdy B - T$ for $i=1,2,$ is of type I or type II.

\begin{prop}
Suppose $T$ is an arborescent tangle and $A$ is an incompressible annulus of type I in $X(T)$.  If $A$ is of type I-A then $A$ is parallel to the annulus $A' \cont \bdy B-T$ with $\bdy A' =\bdy A$.  If $A$ is of type I-B then $A = \bdy N(t_i)$ for some string $t_i \in T$.  In particular, $X(T)$ contains no essential annulus of type I.
\label{prop: type I annulus}
\end{prop}

\begin{proof}
We proceed by induction on the length of the tangle.  Suppose $T=(t_1 \cup t_2)$ is a rational tangle and $A$ is an incompressible annulus of type I in $X(T)$.  Define $P_1$ and $P_2$ as the respective left and right disks of the boundary sphere $\bdy B$.   

Suppose $A$ is a type I-A annulus with $\bdy_1 A, \  \bdy_2 A \cont P_1$.  Consider the disk $A \cup D$ such that $D \cont B,$ $\bdy D = \bdy_1 A$ and ${\rm Int} D \cap A = \emptyset$.  Push $D$ to ${\rm Int} B$ to get a disk $D_1 \cong A \cup D$ which intersects $T$ at a single point.  Since $T$ is trivial, $D_1$ cuts off a ball $B_1$ such that $B_1 \cap T$ is a single unknotted arc $\tau$.  Thus after removing a regular neighborhood of $\tau,$ we get a solid torus bounded by $A \cup A'$ where $A'$ is an annulus on $\bdy B$.  Hence $A \parallel A'$.

If $A$ is of type I-B, then by definition, $\bdy_i A$ bounds a disk $D_i \cont \bdy B$ that intersects the tangle in exactly one point for $i=1,2$.  Thus $A \cup D_1 \cup D_2$ is a sphere intersecting $T$ in exactly two points.  It must be the case that $D_1$ and $D_2$ intersect the tangle in a single point from the same string, $t$.  Since $T$ is rational, the string $t$ is unknotted and therefore $A = \bdy N(t)$.   

Suppose $T = T_1 +_{\Delta} T_2$ is a nontrivial sum and the result holds for any incompressible annulus of type I in $E(T_i)$ for $i=1,2$.  Suppose $A$ is an incompressible annulus of type I in $E(T)$.  Let $P=P(\Delta)$ and consider the intersection $A \cap P$ with minimal number of components.  If $A \cap P =\emptyset$ then the result follows by induction.  Suppose $A \cap P \neq \emptyset$.  By Lemma \ref{lem: can't have intersections of types I and II}, $P$ cuts $A$ into type I annulus components $A_1,A_2,...,A_n$.  If $A_i$ is a type I-A component for some $i \in \{1,2,...,n\},$ then by induction, $A_i$ is parallel to an annulus on $P,$ hence we may reduce the intersection by an isotopy, contradicting the minimality of $|A \cap P|$.  Thus $A_i$ is a type I-B component for $i=1,2,...,n$.  By induction, $A_i=\bdy N(t_i)$ for some string $t_i$ in $T_1$ or $T_2$.  Gluing the components back together, the result follows.
\end{proof}

In particular, Proposition \ref{prop: type I annulus} tells us that for an arborescent tangle $T,$ any essential annulus $A$ in $E(T)$ with $\bdy A \cont \bdy B$ must be a type II annulus.

Although $T(1/2, 1/2)=T[1/2] +_{\Delta} T[1/2]$ is equivalent to $\thh=T[1/2] +_{\Delta} T[-1/2]$ up to isotopy rel $\Delta,$ it is important to recognize some special properties of the latter tangle.  In particular, we want to develop a new notation to describe the sum:  $T_1 +_{\Delta'} T_2$ where $T_1=T_2=\thh$ and ${\Delta'}$ is the bottom disk of $B_1$ and the top disk of $B_2$.  We denote this sum by $\thh \ast \thh,$ i.e., this is the tangle where the $\thh$ tangle is glued on top of another  $\thh$ tangle.  In general, we call such a new tangle $T_1 \ast T_2$ the \emph{product} of tangles $T_1$ and $T_2$.  See Figure \ref{fig: specialSum}.  Note that Kauffman and Lambropoulou \cite{KauffmanL} use this notation and terminology in combining rational tangles.

\begin{figure}[hbtp]
\centering
\includegraphics[width=2in]{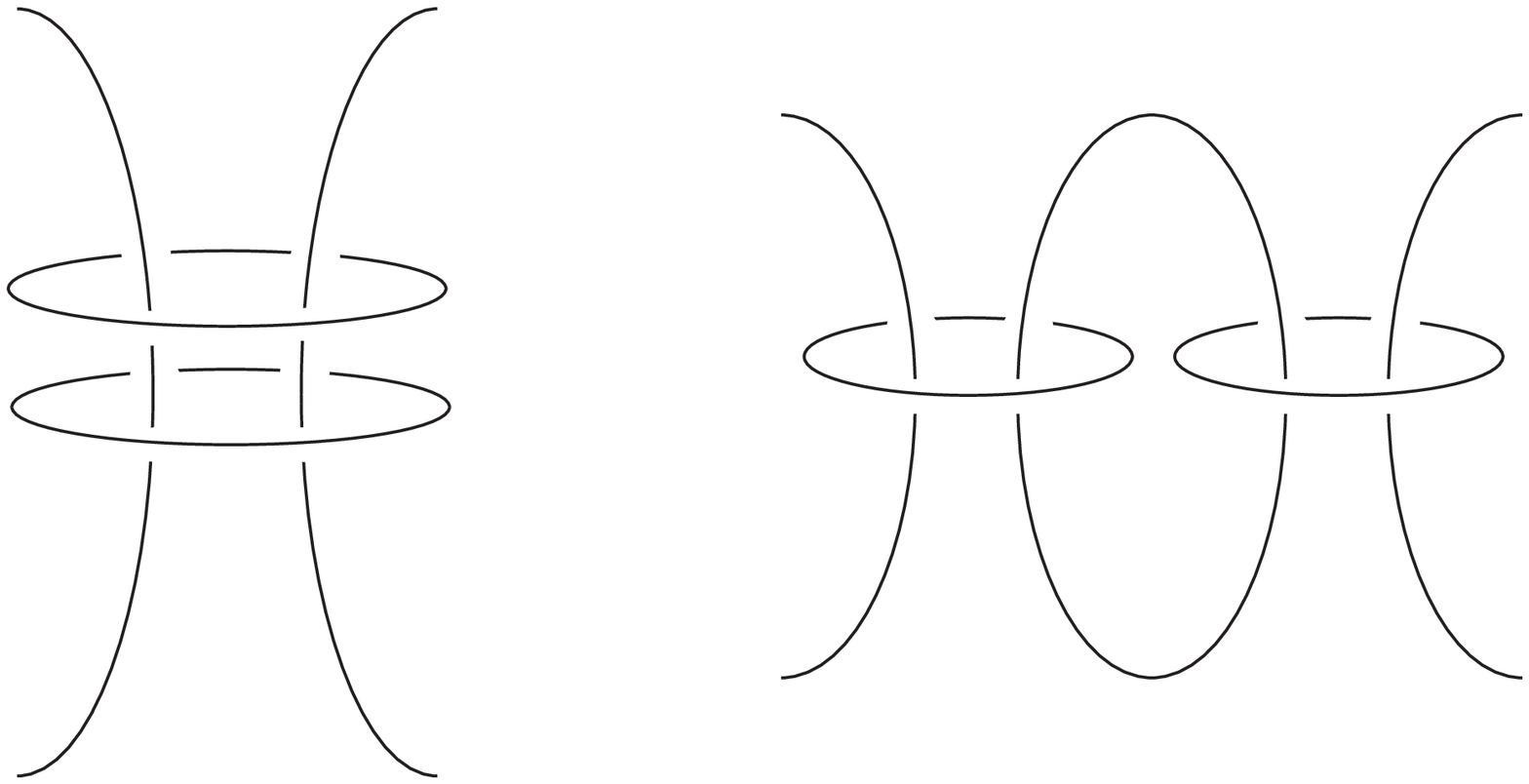}
\caption{$\thh \ast \thh$ versus the Montesinos tangle $\thh +\thh=T(1/2, -1/2, 1/2, -1/2)$}
\label{fig: specialSum}
\end{figure}

This leads to some new terms:

\begin{defin}
\label{def: Q_n tangle, std torus, etc}
  (1) For $n \geq 0$, define $Q_0$ to be the trivial tangle
  $T[\infty]$, $Q_1=T(1/2,-1/2)$, and $Q_n=Q_1 \ast Q_{n-1}$.  Thus
  $Q_n$ is the tangle with two vertical strings and $n$ parallel
  horizontal circle components $c_i$, $i=1,...,n$.  For each $c_i$,
  there exists a horizontal annulus $A_i$ such that $A_i \cap T =
  \bdy_0 A_i=c_i$, and $\bdy_1 A_i \cont (\bdy B - Q_n)$, where $(\bdy
  B - Q_n)$ is the punctured boundary sphere for the tangle $Q_n$.
  See Figure \ref{fig: Q3 and annulus and torus} for an example.

\begin{figure}[hbtp]
  \centering
  \includegraphics[width=4in]{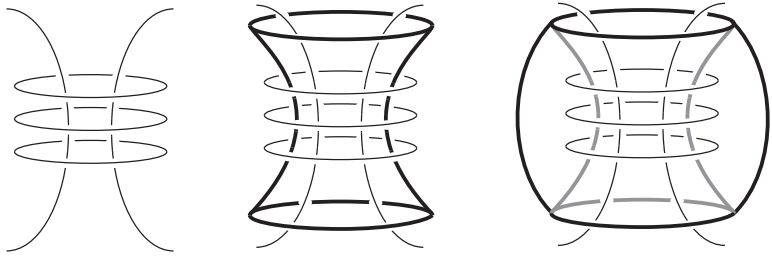}
  \caption{The diagram on the left demonstrates the $Q_3$ tangle.  The
    bold annulus in the middle is the standard annulus for $Q_3$ while
    the diagram on the right demonstrates the standard torus for
    $Q_3$.}
\label{fig: Q3 and annulus and torus}
\end{figure}

(2) For $n \geq 1$, define the \emph{standard annulus} in $Q_n$ as
  the annulus in $Q_n$ which separates the circles from the arcs of
  $Q_n$, as in Figure \ref{fig: Q3 and annulus and torus}.

  (3) For $n \geq 1$, let $A$ be the standard annulus in $Q_n$.  Let
  $A'$ be the annulus on $\bdy B$ with $\bdy A'=\bdy A$.  Define the
  \emph{standard torus} in $Q_n$ as the torus obtained by pushing $A
  \cup A'$ into the interior of $X(Q_n)$.  (Note for $n=1$ this torus
  is inessential since it cuts off a cups in $X(T)$.) See Figure
  \ref{fig: Q3 and annulus and torus} for an example.

(4) Since $T \ast Q_1 = Q_1 \ast T$ up to an isotopy rel $\bdy B$,
  we define \emph{switching} as changing the order of $T$ and $Q_1$.
  See Figure \ref{fig: switching} for an example.
\end{defin}

\begin{figure}[hbtp]
\centering
\includegraphics[width=2.5in]{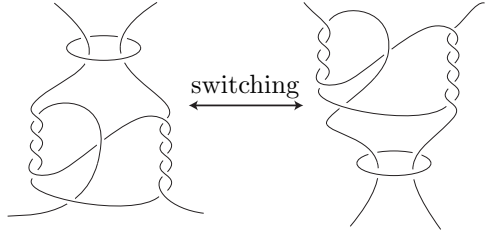}
\put(-111,52){switching}
\caption{The tangle on the left and the tangle on the right are equivalent by switching.}
\label{fig: switching}
\end{figure}

\begin{lem}
If $A$ is a type II essential annulus in $X(T)$ where \tsum\ is a nontrivial sum, and $A$ intersects $P(\Delta)$ in a minimal nonempty collection of essential arcs on $P(\Delta),$ then $T = T(1/2) +_{\Delta} T(1/2)$ up to isotopy rel $\bdy \Delta$ and $A$ is the standard annulus for $Q_1$.
\label{lem: type II essential annulus/arcs}
\end{lem}

\begin{proof}
By assumption, $A$ intersects $P=P(\Delta)$ in essential arcs on $P$.  Since $P$ is essential by Proposition \ref{prop: P essential and bdy sphere incompressible}, these arcs are also essential on $A$.  Let $S_i=\bdy B_i- T_i,$ and $P_i=\overline{S_i - P},$ $i=1,2$.

Suppose $|A \cap P| > 0$.  Since $P$ is separating, $|A \cap P|$ is even.  Let $\alpha_1, \alpha_2,...,\alpha_n$ be arcs in $A \cap P$ such that $\alpha_j$ is adjacent to $\alpha_{j+1}$ on $A$ for $j=1,2,...,n$ (where we define $\alpha_{n+1}=\alpha_1$).  Then $\alpha_j \cup \alpha_{j+1}$ cuts off a disk $D_j$ from $A$ with  $\bdy D_j=\alpha_j \cup \beta_j \cup \alpha_{j+1} \cup \gamma_j$ where $\bigcup \beta_j = \bdy_1 A,$ $\bigcup \gamma_j = \bdy_2 A$.  Recall $\alpha_j$ and $\alpha_{j+1}$ are essential arcs on $P_i$ for $i=1$ or $2,$ and ${\rm Int} D_j \cap P=\emptyset$.  Thus, $\bdy D_j$ is an essential curve on $S_i$ for $i=1$ or $2$ and $D_j$ is a compressing disk for $S_i$.  Furthermore, $D_j$ intersects $P$ in two arcs.  Without loss of generality, suppose $D_1$ is such a disk in $B_1-T_1$.  By Proposition \ref{prop: P essential and bdy sphere incompressible}, $T_1$ is rational, and by Lemma \ref{lem: pq tangle}, $T_1=T(1/2)$.  Similarly, the disk $D_2$ is a compressing disk for $S_2$ intersecting $P$ in two arcs.  By the same argument, $T_2=T(1/2)$.

Since the compressing disk in $B_i-T_i$ is unique up to isotopy, $D_1, D_3,...D_{n-1}$ are all parallel in $B_1-T_1$ while $D_2,D_4,...,D_n$ are parallel in $B_2-T_2$.  Thus the two arcs of $P \cap A$ which are outermost on $\Delta$ belong to the same disk $D_{2r-1}$ in $A \cap (B_1-T_1),$ and they belong to the same disk $D_{2s}$ in $A \cap (B_2-T_2)$ for some $r,s,$ so $D_{2r-1} \cup D_{2s}$ is a component of $A$.  Since $A$ is connected, $D_{2r-1} \cup D_{2s}=A,$ hence $|A \cap P|=2$.  By construction, one can see that $A$ is the standard annulus for $Q_1$.  
\end{proof}

\subsection{Essential tori}
\label{subsection: essential tori}

\begin{defin}
Given a curve $\alpha$ on $\bdy B$ separating $\bdy B$ into two twice-punctured disks, we say that $T$ is \textbf{$\alpha$-annular} exactly when there exists an essential annulus $A$ with $\bdy A \parallel \alpha$.  Otherwise, we say that $T$ is \textbf{$\alpha$-anannular}.
\end{defin}

\begin{lem}
Suppose $T$ is  an arborescent tangle.

(1)  If $A$ is an inessential annulus of type II in $B-T$ and $\bdy A$ does not bound a disk in $B-T,$ then $A$ is parallel to the annulus on $\bdy B-T$ bounded by $\bdy A$.

(2)  Suppose $F$ is an essential annulus or torus in $E(T),$ \tsum, $\bdy F \cap \bdy \Delta = \emptyset,$ and $|F \cap P(\Delta)|$ is minimal.  Then each component of $F \cap E(T_i)$ is essential.

(3)  Let \tsum, where $T_2$ is \bpann.  Suppose $F$ is an essential torus in $E(T)$.  Up to isotopy, $F \cont E(T_i)$ for $i=1$ or $2$.
\label{lem: three parts to show that essential torus is on one side or other for bpann}
\end{lem}

\begin{proof}
(1):  Let $T$ be an arborescent tangle such that $A$ is an inessential annulus of type II in $B-T$ and $\bdy A$ does not bound a disk in $B-T$.  Suppose $A$ is compressible with compressing disk $C$.  Then $\bdy C$ cuts $A$ into two components, $A_1$ and $A_2$.  Thus $C \cup A_1$ is a compressing disk for $\bdy B-T,$ contradicting the fact that $\bdy A$ does not bound a disk in $B-T$. 

Suppose $A$ is $\bdy$-compressible.  Let $D \cont E(T)$ be the boundary compressing disk with $\bdy D = \alpha \cup \beta,$ $\alpha=D \cap A,$ $\beta=D \cap \bdy(B-T),$ and $\alpha$ is essential in $A$.  Note that $\beta$ must lie in the annulus $A' \cont \bdy(B-T)$ with $\bdy A'=\bdy A$.  In this case, boundary compress $A$ along $D$ to get a disk $D'$ which has inessential boundary, $\alpha ',$ bounding a disk $D''$ on $ \bdy(B-T)$.  Since the tangle complement is irreducible by Proposition \ref{prop: arb tangle complements irred}, the sphere $D' \cup D''$ bounds a ball.  Therefore, $D' \parallel D''$ and hence $A$ is parallel to the annulus on $ \bdy(B-T)$ bounded by $\bdy A$. 

(2): Let $F$ be an essential annulus or torus in $E(T),$ \tsum\ an arborescent tangle, $\bdy F \cap \bdy \Delta = \emptyset,$ and $|F \cap P(\Delta)|$ minimal.  Since $\bdy F \cap \bdy \Delta = \emptyset,$ $F$ intersects $P=P(\Delta)$ in circles.  These circles are essential on both $F$ and $P$ by a standard innermost disk/outermost arc argument.  Thus $P$ cuts $F$ into annulus components $F_{i_j} \cont E(T_i)$ with $\bdy F_{i_j} \cont P$, unless $F$ is an annulus, in which case one of the boundary components of two of the sub-annuli $F_{i_j}$ are contained in $\bdy B-T$ (not on $P$).   
Since $F$ is not boundary-parallel, each component $F_{i_j}$ is a type II annulus in $E(T_i)$.  By Part (1), if any such component $F_{i_j}$ is inessential, we can isotope it just past $P$ to reduce $|F \cap P|,$ contradicting minimality.

(3): Suppose \tsum, with $T_2$  \bpann\ and $F$ an essential torus in $E(T)$.  By Part (2), if $F \cap P \neq \emptyset,$ $P$ cuts $F$ into essential type II annulus components $F_{i_j}$ with $\bdy F_{i_j} \cont P=P(\Delta)$ and $\bdy F_{i_j} \parallel \bdy \Delta$.  This contradicts the fact that $T_2$ is \bpann.  Thus $F \cap P = \emptyset$.
\end{proof}

\begin{lem}
Let \tsum\ be a nontrivial sum of arborescent tangles.  If $T_2$ is \bpann\ then any annulus $A$ with $\bdy A \cont P(\Delta)$ can be isotoped into $E(T_1)$.  Thus if $T_1$ is also \bpann\ then \tsum\ is \bpann.
\label{lem: isotope bpann}
\end{lem}

\begin{proof}
  Since $T$ is a nontrivial sum, it is not a rational tangle.  If $A$
  is inessential then by Lemma \ref{lem: three parts to show that essential torus
  is on one side or other for bpann} (1) it is boundary parallel; 
  since $A$ is a type II annulus and $\bdy A \cap \bdy \Delta = \emptyset$, it
  can be isotoped into $B_2 - T_2$.  

  If $A$ is essential then by Lemma \ref{lem: three parts to show that essential torus
  is on one side or other for bpann} (2), up to isotopy we may assume
  that each component of $A \cap X(T_i)$ is essential, but since $T_2$
  is \bpann, no such component exists in $X(T_2)$.  Therefore
  $A\subset X(T_1)$.

  If $T_1$ is \bpann, then the annulus $A$ must be parallel in
  $X(T_1)$ to the annulus on $\bdy X(T_1)$ bounded by $\bdy A$.
  Hence $T$ is \bpann.
\end{proof}

\begin{lem}
\label{lem: only ess annulus in Q1 is the std annulus}
The only type II essential annulus in $B - Q_1$ is the standard annulus. 
\end{lem}

\begin{proof}
Let $A$ be an essential annulus in $B - Q_1$, where $Q_1=T_1 +_{\Delta} T_2$, $T_i=T(1/2)$ for $i=1,2$, and $P=P(\Delta)$.  Suppose $A$ intersects $P$ transversely and the number of components $|A \cap P|$ is minimal.  By a standard innermost circle/outermost arc argument, the components of $A \cap P$ are either all circles or all arcs, essential on both $A$ and $P$.  If the latter is true then we can apply Lemma \ref{lem: type II essential annulus/arcs} to get the result.

Assume that $A \cap P$ is all circles, thus $\bdy A$ is disjoint from $\bdy \Delta$.  Let $A'$ be an annulus in $B$ with $\bdy_1 A' = A' \cap T$ the circle component of $Q_1$, and $\bdy_2 A'$ the horizontal circle on $\bdy B$.  Since $A$ is of type II, the two components in $\bdy A$ are parallel to the vertical circle $\bdy \Delta$, so $\bdy A
\cap \bdy_2 A' \neq \emptyset$.  Suppose the number of components $|A \cap A'|$ is minimal, and denote by $C$ the arc components of $A \cap A'$.

If there exists an arc component $\xi \subset C$ such that $\xi$ is an
inessential outermost arc on $A'$, then $\xi$ cuts off an outermost
disk $X$ from $A'$.  If $\xi$ is essential on $A$, then $X$ is a
$\bdy$-compressing disk for $A$, contradicting the fact that $A$ is
essential.  If $\xi$ is inessential then as above we may reduce the
number of components in the intersection $|A \cap A'|$.  Thus we may
assume that $C$ consists of essential arcs on $A'$.  On the other
hand, since $A \cap \bdy_1 A' = \emptyset$ and $\bdy A \cap \bdy_2 A'
\neq \emptyset$, $C$ is nonempty, and each component of $C$ has both
endpoints on $\bdy_2 A'$ and hence is inessential on $A'$, which is a
contradiction.
\end{proof}

\begin{lem}
  Given the $Q_n$ tangle, $n \geq 1$ and any type II essential annulus
  $A \cont B - Q_n$, there is an isotopy of $Q_n$ so that $Q_n=Q_m
  \ast Q_{n-m}$ and $A$ is the standard annulus in $Q_m$.
\label{lem: Qn annulus standard}
\end{lem}

\begin{proof}
  Let $n \geq 1$ and suppose $A \cont B - Q_n$ is an essential annulus
  of type II with $\bdy A \cont \bdy B -Q_n$.  Let the $n$ core circle
  components of $Q_n$ be represented by $c_i$, $i=1,2,..,n$.  Recall
  that for each $c_i$, there exists an annulus $A_i$ such that $\bdy_0
  A_i=c_i$, $\bdy_1 A_i \cont (\bdy B - Q_n)$, where $(\bdy B - Q_n)$
  is the boundary sphere for the tangle $Q_n$.  

By an innermost circle/outermost arc argument we may assume that $(\cup A_i) \cap A$ consists of circles essential on both $A$ and $\cup A_i$.  (See Figure \ref{fig: messy annulus}.) By cutting and pasting along the annulus cut off by an outermost circle on $A$, we may further assume that $(\cup A_i) \cap A = \emptyset$.  Since $A$ is of type II, the two arc components of $Q_n$ must be on the same component $W$ of $B - \Int N(A)$, which is homeomorphic to $D^2 \times I$ with $A$ identified to $\partial D^2 \times I$.  Since the arc components of $Q_n$ are unknotted in $B$, we see that the other component $X$ of $B - \Int N(A)$ is a solid torus, with $A$ a longitudinal annulus.


Recall that the set of annuli $A_i$ from the above argument do not
intersect $A$ and furthermore, $\bdy_0 A_i =c_i$ and $\bdy_1 A_i \cont
(\bdy B-Q_n)$.  One can view $\bigcup_i \bdy_1 A_i$ as a set of nested
circles with $\bdy A$ separating them into the two groups (see Figure
\ref{fig: nested circles}).

The two boundary components of $A$ break $\bdy B-Q_n$ into an annulus
and two (twice-punctured) disks.  After renumbering, we may assume
that $\bdy_1 A_i$ ($i=1,...,m$) are contained in the annulus part (as
in Figure \ref{fig: nested circles}).

Consequently, the other $(n-m)$ nested circles correspond to the
circle components in the two (twice-punctured) disk portions of $\bdy
B -Q_n$.  By adding a copy of $A$ and doing an isotopy, we may assume
these $(n-m)$ circles all lie on the bottom disk (as in Figure
\ref{fig: nested circles}).  Similarly, for the $m$ core circles
described above, we know that these circles are parallel to circles on the
boundary sphere $\bdy B - Q_n$, with the parallelisms given by the
$A_i$'s.  Using these parallelisms, we can perform isotopy to reorder
the $c_i$'s, $i=1,...,m$ and see that $A$ is the standard annulus for
$Q_m \cont Q_n$.
\end{proof}

\begin{figure}[hbt]
\centering
\includegraphics[width=3in]{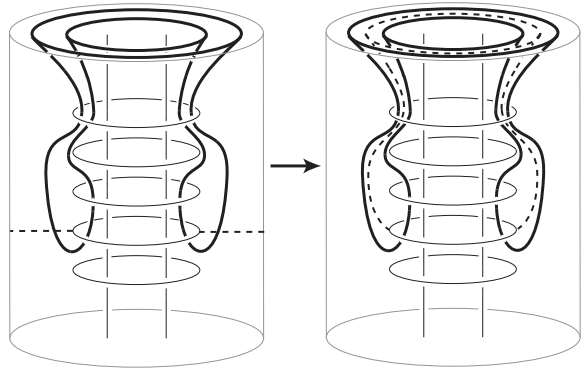}
\caption{This shows an example of how to change $A_4$ (dotted) so that
  its intersection with $A$ (bold) has less components as in the proof
  of Lemma \ref{lem: Qn annulus standard}.}
\label{fig: messy annulus}
\end{figure}

\begin{figure}[hbtp]
\centering
\includegraphics[width=3.5in]{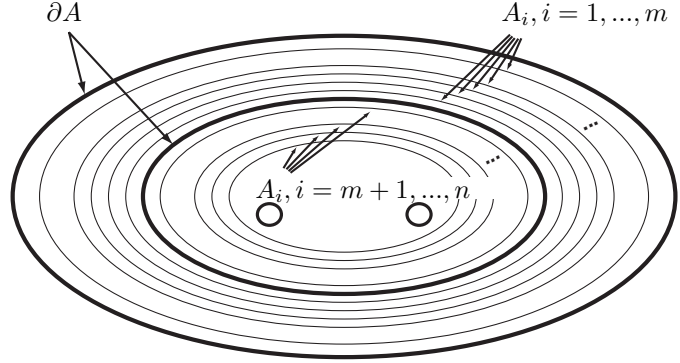}
\put(-239, 128){$\bdy A$}
\put(-67, 128){$A_i, i=1,...,m$}
\put(-160, 61){$A_i, i=m+1,...,n$}
\caption{A diagram showing nested circles as in the proof of Lemma \ref{lem: Qn annulus standard}.}
\label{fig: nested circles}
\end{figure}

\begin{lem}
Let $T$ be an arborescent tangle.  For any type II curve $\alpha \cont \bdy X(T)$, $T$ can be written as $T=Q_n \ast_{\Delta} T'$ such that $\bdy \Delta = \alpha$ and $T'$ is \bpann\ ($n \geq 0$).
\label{lem: can rewrite tangle as sum with Qn}
\end{lem}

\begin{proof}
If $T$ is rational then there is no essential annulus in $E(T),$ thus $T=Q_0 \ast T'$ where $T'=T$ is rational.  Suppose the result holds for any arborescent tangle $T$ having length $\ell(T)\leq n$ and proceed by induction.  Let $T$ be an arborescent tangle such that $\ell(T)=n+1,$ and let $\alpha \cont \bdy B - T$ be a type II curve.  If $T$ is $\alpha$-anannular, then the result holds as $T$ can be written as $T=Q_0 \ast_{\Delta} T'$ with $\bdy \Delta = \alpha$ and $T'=T$.  Suppose, then, that $T$ is $\alpha$-annular with $A$ an essential annulus, $\bdy A \parallel \alpha$.  Note that $A$ must be a type-II annulus by Proposition \ref{prop: type I annulus}.

Write \tsum\ where $\ell(T_i) \leq n$ for $i=1,2$.  If $\bdy \Delta \not \parallel \alpha,$ then $A$ intersects $\Delta$ in arcs.  By Lemma \ref{lem: type II essential annulus/arcs}, $T=T(1/2,1/2)=T(1/2)+_{\Delta}T(1/2)$.  Furthermore, we can write $T=Q_1 \ast_{\Gamma} T'$ where $\bdy \Gamma = \alpha$ and $T'$ is the trivial tangle.

If $\bdy \Delta \parallel \alpha,$ choose a curve $\alpha_i \parallel \alpha$.  By the inductive hypothesis, we may write $T_i = Q_{m_i} \ast_{\Delta_i} T_i',$ $i=1,2,$ where $\bdy \Delta_i=\alpha_i$ and $T_i'$ is $\Delta_i$-anannular.  Now:
\begin{eqnarray*}
T_1 +_{\Delta} T_2 &=& (Q_{m_1} \ast T_1') +_{\Delta} (Q_{m_2} \ast T_2') \mbox{ by definition} \\
&=& (T_1' \ast Q_{m_1}) +_{\Delta} (Q_{m_2} \ast T_2') \mbox{ by switching}
\end{eqnarray*}
However because of how we have chosen $\Delta, \Delta_1, \Delta_2,$ this is the same as $(T_1' \ast Q_{m_1}) \ast_{\Delta} (Q_{m_2} \ast T_2')$
\begin{eqnarray*}
&=& T_1' \ast (Q_{m_1} \ast Q_{m_2}) \ast T_2'  \\
&=& Q_{m_1+m_2} \ast_{\Delta} (T_1' \ast T_2') \mbox{ by switching} \\
&=& Q_{m_1+m_2} \ast_{\Delta} T'' \mbox{ where } T''=T_1' \ast T_2',
\end{eqnarray*}
and by Lemma \ref{lem: isotope bpann}, $T''$ is \bpann.
\end{proof}

\begin{prop}
\label{prop: essential annulus is standard}
Let $T$ be an arborescent tangle.  Then an annulus $A$ in $X(T)$ is essential if and only if
$T=Q_m \ast T'$ for some $m \geq 1$ and $A$ is a standard annulus in $Q_m$.
\end{prop}

\begin{proof}
A standard annulus of $Q_m$ is clearly an essential annulus in $Q_m * T'$, so assume that $A$ is an essential annulus in $X(T)$.  By Proposition \ref{prop: type I annulus}, $A$ is of type II.  By Lemma \ref{lem: can rewrite tangle as sum with Qn} we can write $T$ as $Q_n \ast_{\Delta} T''$, where $T''$ is \bpann, and $\bdy \Delta \parallel \bdy A$.  By Lemma \ref{lem: isotope bpann}, $A \cont X(T)$ can be isotoped into $Q_n$.  By Lemma \ref{lem: Qn annulus standard}, up to isotopy we have $Q_n=Q_m \ast Q_{n-m}$ and $A$ is standard in $Q_m$.  The result now follows by rewriting $T=Q_n \ast T''$ as $Q_m \ast T'$ with $T' = Q_{n-m} \ast T''$.
\end{proof}

\begin{lem}
If $F$ is an incompressible torus in $Q_n,$ $n \geq 1,$ then $Q_n$ is isotopic to $Q_m \ast Q_{n-m}$ such that $F$ is standard in $Q_m,$ $1 \leq m \leq n$.  In particular if $F$ is essential in $X(Q_n)$, then $F$ is standard in $Q_m,$ $2 \leq m \leq n$.
\label{lem: torus standard in Qn}
\end{lem}

\begin{proof}
Let $n \geq 1$ and suppose $F \cont X(Q_n)$ is an incompressible torus.  Define each of the core circle components of $Q_n$ by $c_1, c_2, .., c_n$.  Recall that for each $c_i,$ there exists and annulus $A_i$ such that $\bdy_0 A_i=c_i,$ $\bdy_1 A_i \cont (\bdy B - Q_n),$ where $(\bdy B - Q_n)$ is the boundary sphere for the tangle $Q_n$.  As in the proof of Lemma \ref{lem: Qn annulus standard}, choose the $A_i$ pairwise disjoint and transverse to $F$ with $\sum_i |A_i \cap F|$ minimal.

If $(\bigcup_i A_i) \cap F = \emptyset,$ then $F \cont E(Q_n - A_i)=E(T[\infty]),$ hence $F$ is compressible.  
This contradicts the hypothesis, therefore $(\bigcup_i A_i) \cap F \neq \emptyset$.  
Suppose $F \cap A_1 \neq \emptyset$.  The intersection may only consist of circles which are essential on both $F$ and $A_1$.  Let $\alpha$ be such a circle in the intersection which is ``outermost'' on the annulus with respect to the ball.  (In other words, it cuts off an annulus $A_1' \cont A_1$ such that $F \cap A_1'=\emptyset.)$  Cut the torus along this arc $\alpha$ to get an annulus $F'$ with two copies of $\alpha$ as its boundary components.  Glue each copy of $\alpha$ to one of two copies of $A_1'$ to make $F'$ an annulus with two parallel copies of $\bdy_1 A_1 \cont (\bdy B - Q_n)$ as its boundary.  Since $F$ is incompressible, the new annulus $F'$ must be essential and moreover, $F'$ is an essential type-II annulus.  Lemma \ref{lem: Qn annulus standard} tells us that up to isotopy, $F'$ is a standard annulus in $Q_m,$ $m \leq n$.  Gluing the two copies of $\alpha$ back together we recover $F,$ a standard torus in $Q_m,$ $1 \leq m \leq n$ (since $m=0$ implies that the torus is compressible).  Furthermore, if $F$ is essential in $X(Q_n)$, then $2 \leq m \leq n$.
\end{proof}

\begin{prop}
Let $T$ be an arborescent tangle.  Then a torus $F$ in $X(T)$ is essential if and only if $F$ contains $Q_m$ for some $m \geq 2$ and $F$ is a standard torus in $Q_m$. 
\label{prop: toroidal tangle contains Q2 with torus standard}
\end{prop}

\begin{proof}
First assume that $T$ contains $Q_m$ with $m \geq 2$ and $F$ is a standard torus in $Q_m$.  Then the solid torus $V$ in $B$ cut off by $F$ contains $m \geq 2$ circle components of $T$, which are the cores of $V$; hence $V-T$ is not a cusp and $F$ is incompressible in $V-T$.  If $F$ is compressible in $B-\Int V$, then after compression $F$ would become a reducing sphere of $B-T$, contradicting Proposition \ref{prop: arb tangle complements irred}.  Therefore, $F$ is an essential torus in $X(T)$.  

We now assume that $F$ is an essential torus in $X(T)$ and proceed by induction on the length of the tangle $\ell(T))$.  Suppose $T$ is a tangle having length $\ell (T)=1$.  $T$ is rational and hence atoroidal so the result is vacuously true.  Suppose the result holds for a tangle $T$ having length $\ell(T) \leq k$.

Let $T$ be a tangle of length $k+1$.  By Lemma \ref{lem: can rewrite tangle as sum with Qn}, $T$ can be written as $T=Q_s \ast_{\Delta} T'$ with $T'$ \bpann\ and $P=P(\Delta)$.  If $s > 0,$ then Lemmas \ref{lem: three parts to show that essential torus is on one side or other for bpann} Part (3) and \ref{lem: torus standard in Qn} imply that $F$ is standard in $Q_n$ or $F \cont X(T')$.  Furthermore, $\ell(T') \leq k$ so in the latter case, the inductive hypothesis gives the result.  

Suppose $s=0$ and moreover that $T$ cannot be written in such a sum with $s > 0$. 
Write $T=T_1 +_{\Delta'} T_2$ with $\ell(T_i) \leq k$ for $i=1,2$.  Let $\alpha=\bdy \Delta'$ and apply Lemma \ref{lem: can rewrite tangle as sum with Qn}.  The tangle $T_2$ can be written as $Q_m \ast_{\Delta'} T_2'$ where $T_2$ is $\bdy \Delta'$-anannular.  However, if $m > 0,$ then by switching, $T$ can be written in the form $T=Q_s \ast_{\Delta} T'$ with $T'$ \bpann\ and $s >0$ (by letting $s=m$).  Since we assumed that was not the case, $m=0$.  Thus, $T_2$ is $\alpha$-anannular.  Now by Lemma \ref{lem: three parts to show that essential torus is on one side or other for bpann} Part (3), $F \cont X(T_i)$ for $i=1$ or $2$.  The result follows by induction.    
\end{proof}

We have now determined all arborescent tangles whose complement
contains an essential surface which is an $S^2$, $D^2$, annulus or
torus.  These are summarized in the following theorem to determine all
arborescent tangles whose complements are non-hyperbolic.

\begin{thm}
Suppose $T$ is an arborescent tangle.  Then $X(T)$ is non-hyperbolic if
and only if one of the following holds.

(1) $T$ is a rational tangle.

(2) $T = Q_m * T'$ for some $m\geq 1$.

(3) $T$ contains $Q_n$ for some $n\geq 2$.
\label{thm: hyperbolic tangle complements}
\end{thm}

\begin{proof}
By definition, $X(T)$ is non-hyperbolic if and only if it contains an essential surface $F$ which is a disk, sphere, annulus or torus.  These are determined by Propositions \ref{prop: P essential and bdy sphere incompressible}, \ref{prop: arb tangle complements irred}, \ref {prop: essential annulus is standard}, and \ref {prop: toroidal tangle contains Q2 with torus standard}, respectively.
\end{proof}

\begin{addend} 
Suppose $T$ is an arborescent tangle.

(1) $X(T)$ contains no essential $S^2$.

(2) $X(T)$ contains an essential disk $D$ if and only if $T$ is a rational tangle and $D$ is the disk separating the two strings of $T$.

(3) $X(T)$ contains an essential annulus $A$ if and only if $T = Q_m * T'$ for some $m\geq 1$ and $A$ is a standard annulus in $Q_m$.

(4) If $X(T)$ contains an essential torus $F$ then $T$ contains a $Q_m$ for some $m\geq 2$ and $F$ is a standard torus in $Q_m$.
\label{addendum from YQW}
\end{addend}

\begin{proof}
As above, this follows from Propositions \ref{prop: P essential and bdy sphere incompressible}, \ref{prop: arb tangle complements irred}, \ref {prop: essential annulus is standard}, and \ref {prop: toroidal tangle contains Q2 with torus standard}.
\end{proof}

\subsection{Spheres intersecting $T$ at two points}

As an application of the results in the previous sections, we will
show that any sphere intersecting an arborescent tangle transversely
at two points must be trivial in the sense that it bounds a 3-ball
intersecting the tangle at a single trivial arc.  The following lemma
holds for any $n$-string tangle in a 3-ball.

\begin{lem}
\label{lem: sphere intersecting T twice}
Let $(B,T)$ be a tangle, let $S$ be a sphere in $B$, and let $B'$ be the 3-ball in $B$ bounded by $S$.  Suppose $|S \cap T| = 2$ and the component $t$ of $T$ intersecting $S$ is a circle.  If $B' \cap T$ is not a single unknotted string in $B'$, then either $X(T)$ is reducible, or it is toroidal and the boundary of $B' \cup N(t)$ is an essential torus in $X(T)$.
\end{lem}

\begin{proof}
Let $t' = t \cap B'$, and $t'' = t - \Int(t')$.  The union of $B'$ and a regular neighborhood of $t''$ is a solid torus $V$ in $B$. Since a meridian disk of $V$ intersects $t$ at a single point, $t$ is homologically nontrivial in $V$, hence the torus $F = \bdy V$ is incompressible in $V - T$.  If $F$ is compressible in the outside of $V$ then a compression will produce a reducing sphere of $B-T$, hence $X(T)$ is reducible.

Now assume $F$ is incompressible in $B-T$.  Since $\bdy X(T)$ is a punctured sphere, $F$ is not boundary parallel.  Hence either $F$ is an essential torus and we are done, or it bounds a cusp, which means that $V \cap T$ is the core of $V$, so $B'\cap T$ is a single unknotted arc in $B'$, contradicting the assumption.
\end{proof}

\begin{cor}
\label{cor: sphere in B}
If $T$ is an arborescent tangle and $S$ is a sphere in $B$ intersecting $T$ at two points, then the 3-ball $B'$ in $B$ bounded by $S$ intersects $T$ at a single unknotted string.
\end{cor}

\begin{proof} Let $t'$ be the arc component of $T \cap B'$, let $t$ be
  the component of $T$ containing $t'$, and let $t'' = t - \Int(t')$.

  If $t$ is a circle component of $T$ then by Lemma 
  \ref{lem: sphere intersecting T twice}, either $X(T)$ is
  reducible, which contradicts Proposition \ref{prop: arb tangle complements irred}, 
  or the torus $F = \bdy (B' \cup N(t)$ is an essential torus in $X(T)$.  
  By Proposition \ref{prop: toroidal tangle contains Q2 with torus standard},
  $F$ must be the standard torus in $Q_m$ for some $m\geq 2$, which
  contradicts the fact that the solid torus bounded by $F$ has a
  meridian intersecting $T$ at a single point.

  We now assume that $t$ is an arc component of $T$.  Then the
  frontier of $B' \cup N(t'')$ is a type I-B annulus $F$ in $B$.  By
  Proposition \ref{prop: type I annulus}, $F$ cuts off a cusp in
  $X(T)$, which implies that $B' \cap T = t'$ and $t'$ is an
  unknotted string in $B'$.
\end{proof}

  Consider a tangle $T$ with a single closed component $t'$ that
  bounds a disk $D \cont B$ such that $\Int(D)$ intersects $T$
  transversely in a single point.  We call $t'$ an \emph{earring} of
  the tangle $T$.  

\begin{cor}
Arborescent tangles cannot have earrings.
\label{cor: no earrings}
\end{cor}

\begin{proof}
If this is not the case then a regular neighborhood of the disk $D$ described above is a ball whose intersection with $T$ is not a trivial arc since it contains a closed component.  This contradicts Corollary \ref{cor: sphere in B}.
\end{proof}

\section{Hyperbolicity of arborescent link complements}
\label{section: Arborescent links}

As mentioned in the introduction one great accomplishment of the work of Bonahon and Seibenmann \cite{BS} is that they classified all non-hyperbolic arborescent links.  However, their work has remained incomplete and unpublished.  Oertel \cite{Oertel} classified non-hyperbolic Montesinos links.  See Theorem \ref{thm: Oertel's Montesinos hyperbolic links} below.  The main theorem of this section is Theorem \ref{thm: Bonahon-Seibenmann large arborescent links}, which classifies all non-hyperbolic arborescent links of length at least 4.  Together with Oertel's theorem, this gives an alternative proof of Bonahon-Seibenmann's classification theorem.  Another proof of Bonahon-Seibenmann's classification theorem has been obtained recently by Futer and Gueritaud \cite{FuterGueritaud}, using a completely different approach.  

An \emph{arborescent link} $L$ was defined in Section \ref{section: Definitions and Preliminaries} as constructed from an arborescent tangle $(B,T)$ by adding two arcs on $\bdy B$ to connect the boundary points of $T$.  We proceed by recalling the precise definition from Section \ref{section: Definitions and Preliminaries}.  

\bigskip

\noindent \textbf{Definition \ref{defin: arborescent links in terms of tangles}} {\rm (Wu \cite{Wu2})} \textit{A \emph{small arborescent link} is a rational link or a Montesinos link of length $2$ or $3$; these are simply rational tangles connected along a band.  A \emph{large arborescent link} is obtained by gluing two nonrational arborescent tangles, $T_1$ and $T_2$, by an identification map of their boundary spheres (Conway spheres).  A link is an arborescent link if it is either a small arborescent link or a large arborescent link.}

\bigskip
In other words, if an arborescent tangle of length $2$ or $3$ is a Montesinos tangle, we may turn the tangle into an arborescent link by simply connecting the top two strings and the bottom two strings.  (See Figure \ref{fig: Montesinos Link}.)   To be precise, however, note that an arborescent tangle of length $3$ is not necessarily a Montesinos tangle. For example, the tangle in Figure \ref{fig: length} is not a Montesinos tangle, but after closing it appropriately one gets a Montesinos link.   Montesinos links have been studied in detail by Oertel \cite{Oertel}, who called them star links since the tree diagrams (as in Gabai \cite{Gabai}) are star-shaped.  To denote a Montesinos link as in Figure \ref{fig: Montesinos Link}, we write $K(r_1, r_2,...,r_n),$ where $r_i=p_i/q_i$.

The following theorem of Oertel \cite[Corollary 5]{Oertel} determines all
non-hyperbolic Montesinos links.

\begin{thm} {\rm (Oertel)}  Suppose $K$ is a Montesinos link.  $S^3-K$ has complete hyperbolic structure if $K$ is not a torus link, and it is not equivalent to $L(\frac 12, \frac 12, \frac{-1}2, \frac{-1}2)$,  $L(\frac 23, \frac{-1}3, \frac{-1}3)$, $L(\frac 12, \frac{-1}4, \frac{-1}4)$, $L(\frac 12, \frac{-1}3, \frac{-1}6)$ or the mirror image of these links.
\label{thm: Oertel's Montesinos hyperbolic links}
\end{thm}

It should be noted that the term ``torus link'' in the above theorem and in Bonahon-Seibenmann's unpublished manuscript \cite{BS} is not a torus link in the usual sense that it lies on a trivial torus $F$ in $S^3$; instead, it may contain one or both of the cores of the solid tori bounded by $F$.  Bonahon and Seibenmann have given a classification of Montesinos links which are torus links in the above sense.  
For the convenience of the reader, these include $K(1/2, -1/2, 1/q)$ for $q\neq 0$, $K(1/4, -1/2, 1/3)$, and torus knots $(3,4)\equiv K(-1/3,-1/2,1/3)$ and $(3,5)\equiv K(-1/5,1/2,-1/3)$.   Note that $K(1/4, -1/2, 1/3)$ is equivalent to the torus knot $(2,3)$ union the axis linking this torus knot three times.  
See  \cite[Theorem A.8, Appendix]{BS} and Figure \ref{fig: two possibilities}.  

\begin{figure}[hbtp]
\centering
\includegraphics[width=4in]{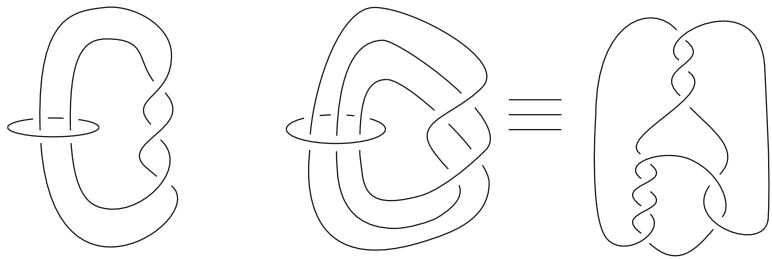}
\caption{On the left is $K(1/2, -1/2, 1/q)$ for $q=3$. On the right is $K(1/4, -1/2, 1/3)$.}
\label{fig: two possibilities}
\end{figure}

The following theorem determines all non-hyperbolic arborescent links
of length at least 4.  Together with the above theorem of Oertel, it
gives an alternative proof of the classification theorem of Bonahon
and Seibenmann for non-hyperbolic arborescent links.  Futer and
Gueritaud \cite{FuterGueritaud} have recently given another proof of
Bonahon-Seibenmann's Theorem using angled structures.

Recall that $Q_2$ denotes the tangle consisting of two vertical arcs
and two horizontal circles.  See Subsection \ref{subsection: std torus and annulus}.

\begin{thm} {\rm (Bonahon-Seibenmann)} 
Suppose $L$ is a large arborescent link.  Then $L$ is non-hyperbolic if and only if it contains $Q_2$.
\label{thm: Bonahon-Seibenmann large arborescent links}
\end{thm}

\begin{proof}  
A useful version of Thurston's Hyperbolization Conjecture is stated in a survey paper by Allen Hatcher \cite{Hatcher1} as follows: \emph{The interior of every compact irreducible atoroidal non-Seifert-fibered 3-manifold whose boundary consists of tori is hyperbolic.}  Thurston's Hyperbolization Conjecture has been proved for Haken manifolds and since the exterior $E(L)$ of a link in $S^3$ has non-empty boundary, it is either reducible or Haken.
If $L$ contains $Q_2$ then either the standard torus $F$ in $Q_2$ is essential in $E(L)$ and hence $E(L)$ is toroidal, or $F$ is compressible, in which case $E(L)$ is reducible; in either case $E(L)$ is non-hyperbolic.  Note that the exterior is irreducible and atoroidal if and only if $X(L)$ is irreducible and atoroidal.  Therefore we need only show that if $L$ is a large arborescent link and if it does not contain $Q_2$ then the complement of $L$ is irreducible and atoroidal, and the exterior is not a Seifert fibered space.  These will be proved in Lemmas \ref{lem: arb link comp irreducible}, \ref{lem: toroidal L contains Q2} and \ref{lem: Non Seifert fibered} below.
\end{proof}

\begin{lem}
 Large arborescent link complements are irreducible.
\label{lem: arb link comp irreducible}
\end{lem}

\begin{proof}
Suppose $F$ is an essential sphere in $X(L)$ for a large arborescent link.   Let $L=T_1 \cup_S T_2$, where $T_1$ and $T_2$ are each arborescent tangles of length $\geq 2$, and $S$ is a Conway sphere.  By Proposition \ref{prop: P essential and bdy sphere incompressible}, $S-T = \bdy X(T_i)$ is incompressible in both $X(T_1)$ and $X(T_2)$. Therefore, $S-T$ is incompressible in $X(L) = S^3 - L$.  Hence we can apply Lemma \ref{lem: boundary reducing disk with empty intersection AND sphere}; i.e. we can find an essential sphere $F'$ which does not intersect $S$.  Hence $F' \cont X(T_i)$ for $i=1$ or $2$, however this contradicts Proposition \ref{prop: arb tangle complements irred}.
\end{proof}

\begin{lem}
\label{lem: toroidal L contains Q2}
If $L$ is a large arborescent link, then $X(L)$ contains an essential
torus if and only if $L$ contains $Q_2$.  Furthermore, the essential
torus is standard in $Q_m$ for some $m \geq 2$.
\end{lem}

\begin{proof}
  If $X(L)$ contains $Q_2$ then let $F$ be the standard torus for
  $Q_2$, with $V$ the solid torus bounded by $F$ intersecting $L$ in 2
  core circles of $F$.  Suppose $F$ is compressible.  It cannot be
  compressible on the side containing the two core curves, so suppose
  it is compressible on the other side.  Compressing along a
  compressing disk $D$ gives a sphere.  This sphere bounds a ball
  which intersects the tangle in two disjoint circles, and there are
  also some components of the link outside of the ball; therefore it
  is an essential sphere, contradicting Lemma \ref{lem: arb link comp
    irreducible}.  If $F$ is boundary parallel, then it cuts off a
  cusp.  The cusp cannot be in $\Int (V)$ since there are two core
  circles from $L$ in $V$.  Thus it must be that $F$ cuts off a cusp
  on the other side of $F$.  Suppose $W$ is the solid torus bounded by
  $F$ on the other side.  Note that $W$ contains the two vertical
  string in the definition of $Q_n$, so there is a meridian disk of
  $W$ intersecting $W\cap L$ in two points.  Therefore any meridian
  disk of $W$ must bound a disk intersecting $W\cap L$ an even number
  of times in $W$, hence $F$ cannot bound a cusp on this side either.
  Therefore, $F$ is essential.

  Now suppose $L$ is a large arborescent link such that $X(L)$
  contains an essential torus, $F$.  We may write $L=T_1 +_S T_2$
  where $T_i$ is an arborescent tangle with length $\ell(T_i) \geq 2$
  for $i=1,2$, and $S$ a Conway sphere, chosen so that $S$ intersects
  $F$ transversely, and the number of components, $|S \cap F|$, is
  minimal.  Let $S'$ be the 4-punctured sphere $X(L) \cap S$; note
  that $|S \cap F|=|S' \cap F|$.  If $F \cap S'$ is empty, then $F
  \cont X(T_i)$ for $i=1$ or $2$.  By Proposition \ref{prop: toroidal
    tangle contains Q2 with torus standard}, $T_i$ contains $Q_2$ and
  the torus $F$ is standard in $Q_m$ for some $m \geq 2$, hence $L$
  contains $Q_2$ and $F$ is standard in $Q_m$ also.  Therefore we
  assume that the intersection $F \cap S'$ is nonempty.
 
By Proposition \ref{prop: P essential and bdy sphere incompressible}, $S'$ is incompressible.   Similar to the proof of Lemma \ref{lem: three parts to show that essential torus is on one side or other for bpann} Part (2), each component of $F \cap S'$ is essential in $X(T_i)$ for $i=1,2$; these are annulus components $A_1, ..., A_n \cont F$ with $\bdy A_i$ essential curves on both $F$ and $S'$ (parallel on $F$).   Thus they must be type-II annulus components by Proposition \ref{prop: type I annulus}.  (Note that $n$ is even.)  Without loss of generality, assume that $A_j \cont X(T_1)$ for $j$ odd, and  $A_j \cont X(T_2)$ for $j$ even.  By Proposition \ref{prop: essential annulus is standard}, for odd $j$, we may write $T_1=Q_{m_j} \ast T_j'$ for some $m_j \geq 1$, and $A_j$ the standard annulus in $Q_{m_j}$.  Similarly for even $j$, we may write $T_2=Q_{m_j} \ast T_j'$ for some $m_j \geq 1$, and $A_j$ the standard annulus in $Q_{m_j}$.   

If $n=2$, then $A_1$ is the standard annulus for $Q_{m_1}$ in $B_1$, $m_1 \geq 1$, and $A_2$ is the standard annulus for $Q_{m_2}$ in $B_2$, $m_2 \geq 1$.  Gluing them together, $F$ is the standard torus for $Q_{m_1+m_2}$, and $m_1 +m_2 \geq 2$.  

Suppose, however, that $n >2$, i.e., $|S' \cap F| > 2$.  If the annuli on both sides of $S$ are nested as in Figure \ref{fig: nested annuli}, then numbering the components from the ``inside-out"  we have $1, 2, ..., k-1, k, k, k-1, ..., 2, 1$.  Gluing the $k$ annuli on the $B_1$ side to the $k$ annuli on the $B_2$ side of $S$, we see that $F$ has more than one component, contradicting the fact that $F$ is a torus.  Thus, without loss of generality, the annuli $A_j$ ($j$ odd) in $B_1$ are not all nested (as in Figure \ref{fig: not nested annuli}).  Hence some $A_j$, say $A_1$, is an ``innermost" annulus and we may isotope the annulus $A_1$ (thus pulling the closed components from $Q_{m_1}$ into $B_2$ past $S$, reducing the number of components in the intersection, $|F \cap S'|$.  Furthermore, since not all the $A_j$ are nested, there is still another annulus, say $A_3$, in $B_1$ which is the standard annulus for $Q_{m_3}$, $m_3 \geq 1$.  This contradicts the fact that we chose $S$ with $|S \cap F|$ minimal.  
\end{proof}

\begin{figure}[hbtp]
\centering
\includegraphics[width=2.5in]{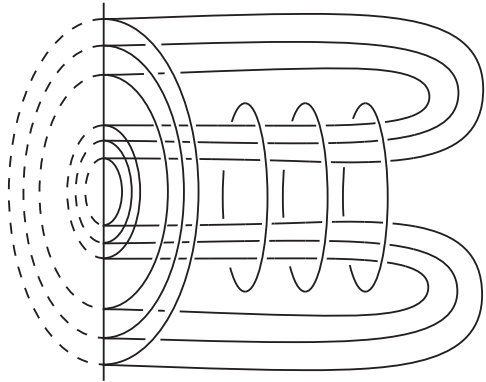}
\put(-175,178){$S$}
\caption{Nested annuli in the proof of Lemma \ref{lem: toroidal L contains Q2}.}
\label{fig: nested annuli}
\end{figure}

\begin{figure}[hbtp]
\centering
\includegraphics[width=2.5in]{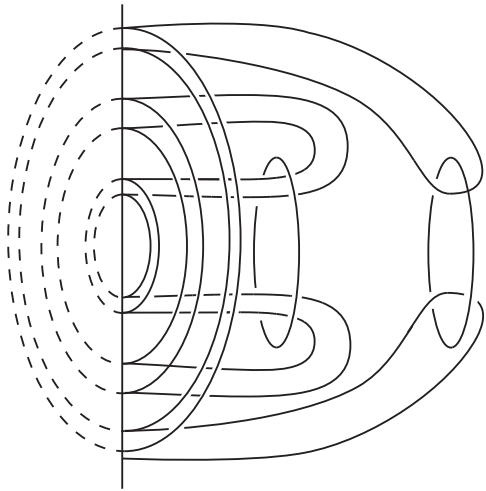}
\put(-168,225){$S$}
\caption{Non-nested annuli for the proof of Lemma \ref{lem: toroidal L contains Q2}.}
\label{fig: not nested annuli}
\end{figure}

\begin{lem}
If $L$ is a large arborescent link, then $E(L)$ is not Seifert fibered.
\label{lem: Non Seifert fibered}
\end{lem}

\begin{proof}
Suppose $L=T_1 \cup_S T_2$, where $S$ is a Conway sphere and $\ell(T_i) \geq 2$ for $i=1,2$, and suppose $E(L)$ is Seifert fibered.  Let $F = S \cap E(L)$.  By Proposition \ref{prop: P essential and bdy sphere incompressible}, $F$ is incompressible in $E(T_i)$.  
Suppose $F$ is $\bdy$-compressible.  Then there exists a disk $D \cont E(L)$ such that $\bdy D =\alpha \cup \beta$ where $\alpha= D \cap F$, $\beta = D \cap \bdy E(L)$, and $\alpha$ is
essential in $F$.  Thus, $\beta$ must run along $\bdy N(t)$ for a string $t \in T_i$ for $i=1$ or $2$, and $\alpha$ is an essential arc on $F$.  Hence the string $t$ is parallel to an arc on $S$ and therefore $T_i$ is a rational tangle.  This contradicts the fact that $\ell(T_i) \geq 2$.  Thus $F$ is not $\bdy$-compressible, and $F$ is essential in $M$.  
By \cite[Proposition 1.11]{Hatcher2}, $F$ must be a vertical or horizontal surface in the Seifert-fibered manifold $E(L)$.  Since vertical surfaces can only be annuli, tori, or Klein bottles (see Hatcher \cite{Hatcher2}), $F$ must be a horizontal surface.

Next, notice that $F$ is a separating (horizontal) surface and hence cuts $E(L)$ into I-bundles.  Filling in the regular neighborhood of the strings in $T_i$ for $i=1$ or $2$ (these are simply I-bundles) on one side of $F$ gives an I-bundle with boundary a sphere.  This can only be an I-bundle over $\mathbb{RP}^2$, which has homotopy type the same as $\mathbb{RP}^2$.  On the other hand, filling the regular neighborhood of the strings back into $E(T_i)$ gives a 3-ball.  This is impossible since a 3-ball and $\mathbb{RP}^2$ have different homotopy types.  Hence $E(L)$ cannot be Seifert fibered.
\end{proof}


\begin{thm}
\label{thm: large arborescent link is prime}
Let $L$ be a large arborescent link.  Then $L$ is a non-split prime link. 
\end{thm}

\begin{proof} 
Let $L$ be a large arborescent link such that $L=T_1 \cup_S T_2$, where $T_i$ is an arborescent tangle with $\ell(T_i) \geq 2$, and $S$ a Conway sphere.  Suppose $B_i$ is the ball in $S^3$ with $\bdy B_i=S$ and $B_i \cap L=T_i$ for $i=1,2$.  By Lemma \ref{lem: arb link comp irreducible}, $L$ is a non-split link.
We must show that $L$ is prime.  Suppose $L=L_1 \# L_2$, where $L_i$ is nontrivial and let $F$ be a decomposing sphere for $L$ with $L \cap F=\{p_1, p_2\}$.  We may assume that $F$ is transverse to $S$, and that $F \cap S$ consists of simple closed curves.  Furthermore, we assume that $F$ has been chosen so that the number of components $|F \cap S|$ is minimal.  If $|F \cap S|=0$, then $F \cont B_i$ for $i=1$ or $2$.  Then by Corollary \ref{cor: no earrings}, $L_i$ is trivial, a contradiction.  

Suppose $|F \cap S|\neq 0$.  Let $\alpha \in F\cap S$ such that $\alpha$ is an innermost curve on $F$.  Then $\alpha$ bounds an innermost disk $D$ on $F$.  We may choose $\alpha$ so that $D \cap L=\emptyset$ or $D \cap L=p_i$ for $i=1$ or $2$ since $F$ is a sphere.  If $D \cap L=\emptyset$, then we may reduce $|F \cap S|$, contradicting minimality.  If $D$ intersects $L$ in a single point, then $\alpha$ bounds a disk $D'$ on $S$ which also intersects $L$ in a single point.  By Corollary \ref{cor: no earrings}, $D \cup D'$ is a sphere which bounds a ball intersecting $T_1$ in an unknotted string.  Thus we may reduce $|F \cap S|$ by an isotopy, contradicting the minimality of $|F \cap S|$.  Thus it must be the case that $|F \cap S|=0$.
\end{proof}

The Hopf link is an arborescent link since it is simply the boundary of a band with 2 twists, or equivalently, the integral rational tangle $T(2)$ with numerator closure.  By the definition of earring, either closed component can be called an earring of the link.  
However a large arborescent link cannot have an earring.

\begin{cor}
 If $L$ is a large arborescent link then $L$ cannot have earrings.
\label{lem: no earrings, links}
\end{cor}

\begin{proof}
Suppose $L$ is a large arborescent link with an earring. A regular neighborhood of the earring is a ball, $B'$, whose boundary intersects the tangle in two points, but $B' \cap L $ is nontrivial.  This contradicts Theorem \ref{thm: large arborescent link is prime}.
\end{proof}

\noindent
I am indebted to Professor Ying-Qing Wu at The University of Iowa for his guidance, support, and thoughtfulness.

\bibliographystyle{amsplain}
\bibliography{refs}

\end{document}